\documentclass[12pt]{article}
\usepackage{amsmath}
\usepackage{amsthm}
\usepackage{amssymb}
\usepackage[T1]{fontenc}
\usepackage[sc]{mathpazo}
\usepackage{enumerate}
\usepackage{etoolbox}

\patchcmd{\section}{\scshape}{\bfseries}{}{}

\newcommand{\R}{{\mathbb{R}}}
\newcommand{\Q}{{\mathbb{Q}}}
\newcommand{\N}{{\mathbb{N}}}
\newcommand{\Z}{{\mathbb{Z}}}

\renewcommand{\subset}{\subseteq}

\DeclareMathOperator{\Frac}{Frac}
\DeclareMathOperator{\tr}{tr}

\newtheorem{theorem}{Theorem}%[section]
\newtheorem{proposition}[theorem]{Proposition}

\newtheorem{corollary}[theorem]{Corollary}

\theoremstyle{definition}

\theoremstyle{remark}
\newtheorem{remark}[theorem]{Remark}
\newtheorem{example}[theorem]{Example}
\newtheorem{counterexample}[theorem]{Counterexample}

\begin{document}
\title{Detecting divisibility in uniquely divisible $\mathfrak{N}$-semigroups 
via homomorphisms to $\mathbb{R}_{>0}$
}
\author{Colin Tan\footnote{\textbf{Email:} 2601203i@student.tp.edu.sg 
\textbf{Subject class:} 
Primary 20M14; Secondary 20M15, 52A05
}}
\date{25 Sep 2026}

%\address
%{Colin Tan,
%Department of Statistics \& Applied Probability,
%National University of Singapore,
%Block S16,
%6 Science Drive 2,
%Singapore 117546}
%\email{statwc@nus.edu.sg}

\maketitle
%\numberwithin{equation}{section}

%-------------------------------------------------------------------%

\begin{abstract}
An \emph{$\mathfrak{N}$-semigroup} is an archimedean, idempotent-free, cancellative and commutative semigroup. For example, the additive semigroup $\R_{>0}$ of strictly positive real numbers is a $\mathfrak{N}$-semigroup. 
Hewitt-Zuckerman showed that every $\mathfrak{N}$-semigroup admits a semigroup homomorphism to $\R_{>0}$.
We show that every uniquely divisible $\mathfrak{N}$-semigroup $S$ admits enough semigroup homomorphisms to $\R_{>0}$ to detect divisibility in $S$. %They were studied by Tamura and Kobayashi.
The main ingredient is a version of the Eidelheit and Kakutani hyperplane separation theorem for vector spaces over $\Q$. A counterexample due to Ravsky explains why this result does not generalise to the larger class of $\mathfrak{N}$-semigroups.
\end{abstract}

%-------------------------------------------------------------------%

The natural numbers $\mathbb{N} = \{1, 2, 3, \dots\}$, the strictly positive rational numbers $\mathbb{Q}_{> 0}$, and the strictly positive real numbers $\mathbb{R}_{> 0}$ under addition are examples of so-called $\mathfrak{N}$-semigroups. All semigroups below are commutative and written additively. Say that $x$ \emph{divides} $y$ in a commutative semigroup $S$ if $x + z = y$ for some $z \in S$, and write $x \mid y$ (this notation is used in \cite[\S 6]{Kobayashi1973}). Say that $S$ is \emph{archimedean} if, for each $x, y \in S$, there is some $n \in \mathbb{N}$ such that $x \mid ny$. An \emph{$\mathfrak{N}$-semigroup} is, by definition, an archimedean, idempotent-free, cancellative, commutative semigroup. $\mathfrak{N}$-semigroups were studied mainly from the 1950s to the 1970s, extensively by Tamura \cite{Tamura1957}, \cite{Tamura1973}, \cite{Tamura1974} and also other mathematicians \cite{Hall1972}, \cite{Hamilton1978}. %Kobayashi showed, for the class of so-called affine $\mathfrak{N}$-semigroups, there are enough homomorphisms to $\R_{> 0}$ to detect the equality binary relation. \textbf{check.} \cite{Kobayashi1973}.

In this note, we observe a local-global principle for divisibility in uniquely divisible $\mathfrak{N}$-semigroups. 
Say that $S$ is \emph{uniquely divisible} if, for each $x \in S$ and $n \in \mathbb{N}$, there exists a unique $y \in S$ such that $ny = x$. %Divisibility and unique divisbility conincide for $\mathfrak{N}$-semigroups. 
For example, $\Q_{>0}$ and $\R_{>0}$ are uniquely divisible, while $\N$ is not. 
%A class of $\mathfrak{N}$-semigroups that are not divisible are the so-called \emph{numerical semigroups}, 
The natural multiplicative $\mathbb{N}$-action on any commutative semigroup $S$ extends uniquely to a $\mathbb{Q}_{> 0}$-action if $S$ is uniquely divisible.  Every $\mathfrak{N}$-semigroup admits a semigroup homomorphism to $\R_{> 0}$, by a result of Hewitt-Zuckerman \cite{HewittZuckerman1956} (see also Kobayashi \cite{Kobayashi1973}). The following theorem says that every uniquely divisible $\mathfrak{N}$-semigroup $S$ admits enough semigroup homomorphisms to $\R_{>0}$ to detect divisibility in $S$.

\begin{theorem}
%[Local-Global Principle for Divisibility] 
\label{thm: local-global}
Let $S$ be a uniquely divisible $\mathfrak{N}$-semigroup. For each $x, y \in S$, if $f(x) \mid f(y)$ in $\mathbb{R}_{> 0}$ for all semigroup homomorphisms $f : S \to \mathbb{R}_{> 0}$, then $x \mid y$ in $S$.
\end{theorem}

The converse of Theorem \ref{thm: local-global} is true, since $x \mid y$ in $S$ implies that $f(x) \mid f(y)$ in $\R_{>0}$ for every semigroup homomorphism $f : S \to \R_{> 0}$. Note that, in $\R_{>0}$, $f(x) \mid f(y)$ is equivalent to strict inequality $f(x) < f(y)$. We give two examples of divisible $\mathfrak{N}$-semigroups to illustrate Theorem \ref{thm: local-global}. 
%Let $\mathbb{K} = \mathbb{Q}$ or $\mathbb{R}$, and let $n \in \mathbb{N}$:

\begin{example}
Consider the $\mathfrak{N}$-semigroup $S = \mathbb{Q}_{>0}^n$, whose elements are $n$-tuples $x = (x_1, \dots, x_n)$, where $x_1, \dots, x_n \in \Q_{>0}$. Each semigroup homomorphism $f : \mathbb{Q}_{>0}^n \to \mathbb{R}_{>0}$ is the dot product against some associated $t \in \mathbb{R}_{\ge 0}^n \setminus \{0\}$. Write $f_t :\mathbb{Q}_{>0}^n \to \mathbb{R}_{>0}$ for the semigroup homomorphism associated to $t \in \mathbb{R}_{\ge 0}^n \setminus \{0\}$, given by $f_t(x) = t \cdot x = t_1 x_1 + \cdots + t_n x_n$ for all $x \in S$.

Now, suppose that $f_t(x) \mid f_t(y)$ in $\R_{> 0}$ for all $t \in \mathbb{R}_{\ge 0}^n \setminus \{0\}$. In particular, choosing $t = e_i = (0, \dots, 0, 1, 0, \dots, 0)$ with a single `$1$` in the $i$-th coordinate, then $x_i = f_{e_i}(x) < f_{e_i}(y) = y_i$. Letting $i$ range from $1, \dots, n$, we obtain $x \mid y$ in ${\Q_{>0}}^n$. 
\end{example}

\begin{example}
Consider $S = \mathcal{P}_n(\Q)$, the additive semigroup of positive-definite symmetric $n \times n$ matrices over $\mathbb{Q}$. Each semigroup homomorphism $f : \mathcal{P}_n(\Q) \to \mathbb{R}_{>0}$ is given by the Frobenius inner product against some nonzero positive semidefinite matrix $A$ over $\mathbb{R}$. Write $f_A : \mathcal{P}_n(\Q) \to \R_{> 0}$ for the semigroup homomorphism associated to nonzero $A \succeq 0$ over $\R$, given by $f_A(M) = \langle A, M\rangle =  \tr(AM)$ for all $M \in S$, where $\tr$ is the trace operator.

Let $M, N \in \mathcal{P}_n(\Q)$. Suppose that $f_A(M) \mid f_A(N)$ in $\R_{> 0}$ for all nonzero $A \succeq 0$ over $\R$. In particular, when $A = v v^\intercal$ has rank $1$, $v^\intercal M v < v^\intercal N v$ for all nonzero column vectors $v \in \R^n$, therefore $N - M$ is positive definite, i.e.\ $M \mid N$ in $\mathcal{P}_n(\Q)$.
\end{example}

Theorem \ref{thm: local-global} follows from a version of the Eidelheit-Kakutani hyperplane separation theorem for vector spaces over $\Q$. 
Let $V$ be a $\Q$-vector space. For $x, y \in V$, write $[x, y]_{\Q} = \{(1 -t) x + ty : \, t \in \Q, 0 \le t \le 1\}$. A \emph{$\Q$-convex set} in $V$ is a subset $S \subset V$ such that $[x, y]_{\Q} \subset S$ for all $x , y\in S$. In a $\Q$-convex set $S \subset V$, a point $u \in S$ is a \emph{$\Q$-algebraic interior point} if, for each $v \in V$, there exists $\varepsilon \in \Q_{> 0}$ with $[u - \varepsilon v, u + \varepsilon v]_{\Q} \subset S$. The above terminology is taken from  \cite[p.\ 375]{ScheidererGTM}. 

A \emph{cone} in $V$ is a subset $C \subset V$ closed under the multiplication of scalars from $\Q_{> 0}$, i.e.\ $cx \in C$ for all $c \in \Q_{>0}$ and all $x \in S$. We do not require our cones to contain the origin. 

The following version of the Eidelheit and Kakutani Theorem for $\Q$-convex cones with $\Q$-algebraic interior point was first noted by Burgdorf-Scheiderer-Schweighofer to follow from the original  Eidelheit and Kakutani Theorem over $\R$ \cite[(Subsection) 2.4]{BurgdorfScheidererSchweighofer2012}. A detailed deduction was given by Scheiderer in his recent textbook  \cite[Appendix B]{ScheidererGTM}. (The original papers of Eidelheit and Kakutani are \cite{Eidelheit1936} and \cite{Kakutani1937}. Expositions of their hyperplane separation theorem over $\R$ are given in textbooks by Barvinok \cite[III.1.7]{Barvinok2002} and K\"oethe \cite[around p.\ 177]{Koethe1969}.)

\begin{proposition}[{\cite[Corollary B.19]{ScheidererGTM}}] \label{prop: separationTheorem}
Let $V$ be a $\Q$-vector space and let $C \subset V$ be a $\Q$-convex cone with
$\Q$-algebraic interior point $u \in C$. If $v \in V$ is not a $\Q$-algebraic interior point of $C$, there is a $\Q$-linear map $f : V \to \R$ such that $f \ge 0$ on $C$, $f (u) = 1$ and $f (v) \le 0$.
\end{proposition}

We need only the special case when $C$ is \emph{$\Q$-algebraically open}, i.e.\ when every point in $C$ is a $\Q$-algebraic interior point.
%and \cite[Definition II.1.5, p.\ 45]{Barvinok2002}. 
A cone is \emph{blunt} if it does not contain the origin.
Any proper $\Q$-algebraically open $\Q$-convex cone is necessarily blunt.

\begin{corollary} \label{cor: separationTheorem}
Let $C \subset V$ be a $\Q$-algebraically open $\Q$-convex cone in a $\Q$-vector space $V$. 
For any $v \in V \setminus C$, there is a $\Q$-linear map $f : V \to \R$ such that $f > 0$ on $C$ and $f(v) \le 0$.  
\end{corollary}

\begin{proof}
If $C = V$, then $V \setminus C = \emptyset$, thus the conclusion is vacuously true.

Suppose instead that the $\Q$-algebraically open $\Q$-convex cone $C \subsetneqq V$ is  proper. Choose some $u \in C$. After applying Proposition \ref{prop: separationTheorem}, the only thing left to show is that $f > 0$ on $C$. Indeed, any $x \in C$ is a $\Q$-algebraic interior point, so $x - \varepsilon u \in C$ for some $\varepsilon \in \Q_{>0}$. Therefore $f(x - \varepsilon u) \ge 0$ and hence $f(x) \ge \varepsilon f(u) = \varepsilon > 0$.
\end{proof}

\begin{proof}[Proof of Theorem \ref{thm: local-global}]
A uniquely divisible $\mathfrak{N}$-semigroup $S$ embeds as a $\Q$-convex cone in its abelian group of fractions $V := \Frac(S)$, which is divisible and torsion-free and hence naturally a $\Q$-vector space. This cone is $\Q$-algebraically open because $S$ is archimedean and uniquely divisible. Indeed, given $x \in S$ and $v = a - b \in V$, where $a, b \in S$, the archimedeanity of $S$ gives some large $n \in \N$ such that $a, b \mid nx$ in $S$. Hence, $\varepsilon a, \varepsilon b \mid x$, where $\varepsilon := 1/n \in \mathbb{Q}_{>0}$, since $S$ is uniquely divisible. 
Therefore $x \pm \varepsilon v = x \pm \varepsilon (a - b) = x \pm \varepsilon a \mp \varepsilon b \in S$, hence $[x - \varepsilon v, x + \varepsilon v]_{\Q} \subset S$ due to its $\Q$-convexity. 

We are ready to prove the contrapositive of Theorem \ref{thm: local-global}.
Suppose that $x \nmid y$ in $S$. Then $y - x \in V \setminus S$. Corollary \ref{cor: separationTheorem} then yields a $\Q$-linear map $f : V \to \R$ such that $f > 0$ on $S$ and $f(y - x) \le 0$, which gives $f(y) \le f(x)$. Restricting $f$ to $S$ yields the desired semigroup homomorphism $f : S \to \R_{> 0}$ witnessing that $f(x) \nmid f(y)$ in $\R_{>0}$.
\end{proof}

% Praise the Lord!

\begin{remark}
The target uniquely-divisible $\mathfrak{N}$-semigroup $\R_{>0}$ in Theorem \ref{thm: local-global} cannot be replaced by $\Q_{>0}$, or the local-global principle fail.
Indeed, $S = \R_{>0}$ would be a counterexample, as there are no semigroup homomorphisms $ \R_{>0} \to \Q_{>0}$. 
\end{remark}

We end with a counterexample, due to Ravsky \cite{Ravsky2025}, building on a comment of Tringali, both of whom kindly discussed a question of the author on MathOverflow. This counterexample tells us that the local-global principle of Theorem \ref{thm: local-global} does not hold more generally for arbitrary $\mathfrak{N}$-semigroups. A commutative semigroup $S$ is \emph{divisible} if, for each $x \in S$ and $n \in \mathbb{N}$, there is some $y \in S$ such that $ny = x$ (see \cite{user69762012}). 

\begin{example}[{\cite{Ravsky2025}}]
There is a $\mathfrak{N}$-semigroup $S$ that is not divisible and some $x,y \in \mathfrak{N}$ such that $x \nmid y$ in $S$, yet $f(x) \mid f(y)$ for every semigroup homomorphism $f : S \to \R_{> 0}$.

For example, writing $\Z_2 = \{0, 1\}$ for the cyclic group of order 2, take
\begin{equation}
S = (\N \oplus \Z_2) \setminus \{(1, 1)\}, 
\end{equation}
and $x = (1, 0)$, $y = (2, 1)$ in $S$.
\end{example}

\begin{proof}
We verify only that $S$ is archimedean. (The reader may routinely check that $S$ satisfies the other properties of a $\mathfrak{N}$-semigroup.)
%The other $3$ properties that define a $\mathfrak{N}$-semigroup -- namely cancellativity, not having idempotents, and commutativity -- are inherited from the corresponding properties satisfied by both $\N$ and $\Z_2$. 
To this end, let $(m,g), (m',g') \in S$ be given. Since the first component $\N \subset \N \oplus \Z_2$ of the direct sum is archimedean, there is some $n \in \N$ such that $m \mid n m'$ in $\N$. Without loss of generality, assume that $n$ is large enough that 
\begin{equation} \label{eq: nLargeEnough}
1 + m < nm';
\end{equation}
for example, set $n =  m + 2 \in \N$. Furthermore $g \mid ng'$ since the abelian group $\Z_2$ is closed under subtraction. Thus $(m, g) \mid (nm',ng')$ in $\N \oplus \Z_2$. Therefore $(m, g) \mid (nm',ng')$ in $S$ also because of \eqref{eq: nLargeEnough}.

Also, $S$ is not divisible since the first component $\N \subset \N \oplus \Z_2$ is not divisible. 

Finally, $x \nmid y=(1,0)$ in $S$, since $y - x = (1, 1)$ does not lie in $S$ by construction. Yet, because $4x = (4, 0) = 2(2,1) = 2y$, thus 
every semigroup homomorphism  $f:S \to \R_{> 0}$ satisfies $4f(x) = 2f(y)$ and hence $f(x) < f(y)$ in $\R_{> 0}$.
\end{proof}

The author would be interested to know whether a separation theorem such as Theorem \ref{thm: local-global} holds true more generally for divisible $\mathfrak{N}$-semigroups that are not necessarily uniquely divisible. 

\paragraph{Acknowledgements.} We thank Mark Kambites and Salvo Tringali for private communication. We thank Salvo Tringali and Alex Ravsky for discussion on MathOverflow.
Chatbots, such as Gemini, ChatGPT, and Claude were used throughout the research process. This paper itself is largely written by hand and the author takes responsibility for its correctness.

%===================================================================%
%===================================================================%

\end{document}